\documentclass[10pt]{article}
\usepackage[left=1in,top=1in,right=1in,bottom=1in,letterpaper]{geometry}

\usepackage{amssymb,amsmath,enumerate}
\usepackage{latexsym,amsfonts,amscd,amsxtra,amstext}
\usepackage{algorithm,algorithmic}

\usepackage{hyperref}
\hypersetup{colorlinks=true, linkcolor=blue, citecolor=blue, urlcolor=blue, pdftitle={On the Convergence of Decentralized Gradient Descent}, pdfauthor={Kun Yuan, Qing Ling, Wotao Yin}}

\usepackage[pdftex]{graphicx}

\usepackage[usenames]{color}

\usepackage[color,final]{showkeys}

\usepackage{amsmath,amssymb,amsthm}
\usepackage{latexsym,amsfonts,amscd,amsxtra,amstext}
\usepackage{algorithm,algorithmic}
\usepackage{url}

\usepackage{epsfig} 
\usepackage{threeparttable}
\usepackage{index}
\usepackage{subfigure}

\usepackage[normalem]{ulem} 

\newcommand{\cX}{{\mathcal{X}}}

\newcommand{\RR}{\mathbb{R}} 
\newcommand{\sign}{\mathrm{sign}} 

\newcommand{\dom}{{\mathrm{dom}}} 
\newcommand{\grad}{{\nabla}}    
\newcommand{\Proj}{{\mathrm{Proj}}} 

\DeclareMathOperator*{\Min}{minimize}

\newcommand{\st}{\mbox{subject to}}

\newcommand{\bc}{\begin{center}}
\newcommand{\ec}{\end{center}}

\newcommand{\bdm}{\begin{displaymath}}
\newcommand{\edm}{\end{displaymath}}

\newcommand{\beq}{\begin{equation}}
\newcommand{\eeq}{\end{equation}}

\newcommand{\bfl}{\begin{flushleft}}
\newcommand{\efl}{\end{flushleft}}

\newcommand{\bt}{\begin{tabbing}}
\newcommand{\et}{\end{tabbing}}

\newcommand{\beqn}{\begin{eqnarray}}
\newcommand{\eeqn}{\end{eqnarray}}

\newcommand{\beqs}{\begin{align*}} 
\newcommand{\eeqs}{\end{align*}}  

\newtheorem{theorem}{Theorem}

\newtheorem{assumption}{Assumption}

\newtheorem{corollary}{Corollary}
\newtheorem{remark}{Remark}
\newtheorem{lemma}{Lemma}

\begin{document}

\title{On the Convergence of Decentralized Gradient Descent}
\date{}
\author{Kun Yuan$^*$\and Qing Ling\footnote{K. Yuan and Q. Ling are with Department of Automation,
University of Science and Technology of China, Hefei, Anhui
230026, China. \textup{kunyuan@mail.ustc.edu.cn and
qingling@mail.ustc.edu.cn}} \and Wotao Yin\footnote{W. Yin is with
Department of Mathematics, University of California, Los Angeles,
CA 90095, USA. \textup{wotaoyin@math.ucla.edu}}} \maketitle
\begin{abstract}
Consider the consensus problem of minimizing $f(x)=\sum_{i=1}^n
f_i(x)$, where $x\in\RR^p$ and each $f_i$ is only known to the
individual agent $i$ in a connected network of $n$ agents. To
solve this problem and obtain the solution, all the agents
collaborate with their {neighbors} through information exchange.
This type of decentralized computation does not need a fusion
center, offers better network load balance, and improves data
privacy. This paper studies the decentralized gradient descent
method \cite{Nedic2009}, in which each agent $i$ updates its local
variable $x_{(i)}\in\RR^n$ by {combining} the average of its
neighbors' with a {local} negative-gradient step $-\alpha \nabla
f_i(x_{(i)})$. The method is described by the iteration
\begin{align}
\label{abs:dgd}
x_{(i)}(k+1) \gets \sum_{j=1}^n w_{ij} x_{(j)}(k) -
\alpha \nabla f_i(x_{(i)}(k)),\quad\text{for each agent}~i,
\end{align}
where $w_{ij}$ is nonzero only if {$i$ and $j$} are neighbors
or $i=j$ and the matrix  $W=[w_{ij}] \in \mathbb{R}^{n \times n}$
is symmetric and doubly stochastic.

This paper analyzes the convergence of this iteration and derives
its rate of convergence under the assumption that each $f_i$ is
proper closed convex and lower bounded, $\nabla f_i$ is Lipschitz
continuous with constant $L_{f_i}>0$, and the stepsize $\alpha$ is
fixed. Provided that $\alpha < O(1/L_h)$, where
$L_h=\max_i\{L_{f_i}\}$, the objective errors of all the local
solutions and the network-wide mean solution reduce at rates of
$O(1/k)$ until they reach a level of $O(\alpha)$. If $f_i$ are
(restricted) strongly convex, then all the local solutions and the
mean solution converge to the global minimizer $x^*$ at a linear
rate until reaching an $O(\alpha)$-neighborhood of $x^*$. We also
develop an iteration for decentralized basis pursuit and establish
its linear convergence to an $O(\alpha)$-neighborhood of the true
sparse signal. This analysis reveals how the convergence of
\eqref{abs:dgd} depends on the stepsize, function convexity, and
network spectrum.
\end{abstract}

\section{Introduction}
Consider that $n$ agents form a connected network and
collaboratively solve a consensus optimization problem
\begin{align}
\label{eq:original problem}
\Min\limits_{x\in\RR^p} \quad f(x) = \sum_{i=1}^n f_i(x),
\end{align}
where each $f_i$ is only available to agent $i$. {A pair of agents
can exchange data if and only if they are connected by a direct
communication link; we say that such two agents are neighbors of
each other.} Let $\cX^*$ denote the set of solutions to
\eqref{eq:original problem}, which is assumed to be non-empty, and
let $f^*$ denote the optimal objective value.

The traditional (centralized) gradient descent iteration is
\beq\label{org_grad}
x(k+1) = x(k) - \alpha \nabla f(x(k)),
\eeq
where $\alpha$ is the stepsize, either fixed or varying with  $k$. To apply  iteration \eqref{org_grad} to problem \eqref{eq:original problem} under the decentralized situation, one has two choices of implementation:
\begin{itemize}
\item let a fusion center (which can be a designated agent) carry
out iteration \eqref{org_grad}; \item let all the agents carry out
the same iteration \eqref{org_grad} in parallel.
\end{itemize}
In either way, $f_i$ (and thus $\nabla f_i$) is only known to
agent $i$. Therefore, in order to obtain $ \nabla f(x(k)) =
\sum_{i=1}^n \nabla f_i(x(k))$, every agent $i$ must have  $x(k)$,
compute $\nabla f_i(x(k))$, and then send out $\nabla f_i(x(k))$.
This approach requires synchronizing $x(k)$ and
scattering/collecting $\nabla f_i(x(k))$,  $i=1,\ldots,n$, over
the entire network, which incurs a significant amount of
communication traffic, especially if the network is  large and
sparse. A decentralized approach will be more viable since its
communication is restricted to between neighbors. Although there
is no guarantee that  decentralized algorithms use less
communication (as they tend to take more iterations), they provide
better network load balance and tolerance to the failure of
individual agents. In addition, each agent can keep its $f_i$ and
$\nabla f_i$ private to some extent\footnote{ Neighbors of $i$ may
know the samples of $f_i$ and/or $\nabla f_i$ at some points
through data exchanges and thus obtain an interpolation of
$f_i$.}.

Decentralized gradient descent \cite{Nedic2009} does not rely on a fusion center or network-wide communication. It carries out an approximate version of \eqref{org_grad} in the following fashion:
\begin{itemize}
\item let each agent $i$ hold an approximate \textit{copy}
$x_{(i)}\in\RR^p$ of $x\in\RR^p$;
 \item let each agent $i$ update its
$x_{(i)}$ to the weighted average of its neighborhood; \item let each
agent $i$ apply $-\nabla f_i(x_{(i)})$  to
decrease   $f_i(x_{(i)})$.
\end{itemize}
At each iteration $k$, each agent $i$ performs the following
steps:
\begin{enumerate}
\item computes $\nabla f_i(x_{(i)}(k))$; \item computes the
neighborhood weighted average $x_{(i)}(k+1/2) =  \sum_{j} w_{ij}x_{(j)}(k)$, where
$w_{ij}\not=0$ only if $j$ is a neighbor of $i$ or $j=i$;
\item applies $x_{(i)}(k+1) = x_{(i)}(k+1/2)-\alpha\nabla f_i(x_{(i)}(k))$.
\end{enumerate}
Steps 1 and 2 can be carried out in parallel, and their results
are used in Step 3. Putting the three steps together, we arrive at
our main iteration \beq\label{dec_grad} \boxed{x_{(i)}(k+1) =
\sum_{j=1}^n w_{ij}x_{(j)}(k) - \alpha \nabla
f_i(x_{(i)}(k)),\quad i = 1,2,\ldots, n.} \eeq When $f_i$ is not
differentiable, by replacing $\grad f_i$ with a member of
$\partial f_i$ we obtain the decentralized \emph{subgradient}
method \cite{Nedic2009}. Other decentralization methods are
reviewed {Section \ref{blabla}} below.

We assume that the mixing matrix $W=[w_{ij}]$ is symmetric and
doubly stochastic. The eigenvalues of
 $W$ are real and  sorted in a nonincreasing order $1 =
\lambda_1(W) \geq \lambda_2(W) \geq \cdots \geq \lambda_n(W) \geq
-1$. Let the second largest magnitude of the  eigenvalues of $W$ be
denoted as \beq\label{beta}\beta =
\max\left\{|\lambda_2(W)|,|\lambda_n(W)|\right\}.\eeq
The optimization of
matrix $W$ and, in particular, $\beta$, is not our focus; the reader is referred to \cite{Boyd2004}.

Some basic questions regarding the decentralized gradient
method include: (i) When does $x_{(i)}(k)$ converge? (ii) Does it
converge to $x^* \in \cX^*$? (iii) If $x^*$ is not the limit,
does consensus (i.e., $x_{(i)}(k)=x_{(j)}(k)$, $\forall i,j$) hold asymptotically? (iv) How do the properties of $f_i$ and the
network affect convergence?

\subsection{Background}\label{backgrd}
The study on decentralized optimization can be traced back to the
seminal work in the 1980s \cite{Tsitsiklis1986,Tsitsiklis1984}.
Compared to  optimization with a fusion center that
collects data and performs computation, decentralized
optimization enjoys the advantages of scalability to network
sizes, robustness to dynamic topologies, and privacy preservation
in data-sensitive applications \cite{Sayed2013,
Ling2013_reweighted, Olfati-Saber2007, Yan2013}. These properties are important for
applications where data are collected by distributed
agents, communication to a fusion center is expensive or
impossible, and/or agents tend to keep their raw data private;
such applications arise in wireless sensor networks
\cite{Ling2010, Predd2007, Schizas2008,Zhao2002}, multivehicle
and multirobot networks \cite{Cao2013, Ren2007, Zhou2010}, smart
grids \cite{Giannakis2013, Kekatos2013}, cognitive radio networks
\cite{Bazerque2010, Bazerque2011}, etc. The recent research
interest in big data processing also motivates the work of
decentralized optimization in machine learning \cite{Duchi2012,
Tsianos2012_application}. Furthermore, the decentralized optimization problem \eqref{eq:original problem} can be extended
to  the online or dynamic settings where the objective function becomes
an online regret \cite{Tsianos2012_strong, Yan2013} or a dynamic
cost \cite{Cavalcante2013, Jakubiec2013, Ling2013_dynamic}.

{To demonstrate how decentralized optimization works},
we take spectrum sensing in a cognitive radio network as an
example. Spectrum sensing aims at detecting unused spectrum bands,
{and thus enables} the cognitive radios to opportunistically
use them for data communication. Let $x$ be a vector whose
elements are the signal strengths of spectrum channels. Each
cognitive radio $i$ takes time-domain measurement
 $b_i = F^{-1} G_i x + e_i$, where $G_i$ is
the channel fading matrix, $F^{-1}$ is the
inverse Fourier transform matrix, and $e_i$ is the measurement
noise. To each cognitive radio $i$, assign a local objective function   $f_i(x) = (1/2) \|b_i -
F^{-1} G_i x\|^2$ or the regularized function $f_i(x) = (1/2)
\|b_i - F^{-1} G_i x\|^2 + \phi(x)$, where
$\phi(x)$ promotes a certain structure of $x$. To estimate $x$, a set of geologically nearby cognitive
radios collaboratively solve the consensus optimization problem
\eqref{eq:original problem}. Decentralized
optimization is suitable for this application since  communication between nearby cognitive
radios are  fast and energy-efficient and, if a cognitive radio joins and leaves the network, no
reconfiguration is needed.

\subsection{Related methods}\label{blabla}
The decentralized stochastic subgradient projection algorithm
\cite{Ram2010} handles constrained optimization; the fast
decentralized gradient methods \cite{Jakovetic2013} adopts
Nesterov's acceleration; the distributed online gradient descent
algorithm\footnote{Here we consider its decentralized batch
version.} \cite{Tsianos2012_strong} has nested iterations, where
the inner loop performs a fine search; the dual averaging
subgradient method \cite{Duchi2012} {carries out} a
projection operation after averaging and descending.
Unsurprisingly, decentralized computation tends to require more
assumptions for convergence than similar centralized computation.
All of the above algorithms are analyzed under the assumption of
bounded (sub)gradients. Unbounded gradients can potentially cause
algorithm divergence. When using a fixed stepsize, the above
algorithms (and iteration \eqref{dec_grad} in particular) converge
to a neighborhood of $x^*$ rather than $x^*$ itself. The size of
the neighborhood goes monotonic in the stepsize. Convergence to
$x^*$ can be achieved by using  diminishing stepsizes in
\cite{Duchi2012,Jakovetic2013,Tsianos2012_strong} at the price of
slower rates of convergence. With diminishing stepsizes,
\cite{Jakovetic2013} shows an outer loop complexity of $O(1/k^2)$
under Nesterov's acceleration when the inner loop performs a
substantial search job, without which the rate reduces to
$O(\log(k)/k)$.

\subsection{Contribution and notation}
This paper studies the convergence of iteration \eqref{dec_grad} under the following assumptions.
\begin{assumption}\label{assmp1}
\begin{enumerate}
\item[a)] For $i=1,\ldots,n$, $f_i$ is proper closed convex, lower bounded, and Lipschitz differentiable with constant $L_{f_i}>0$. \item[b)] The network
has a synchronized clock in the sense that \eqref{dec_grad} is
applied to all the agents at the same time intervals, the network is
connected, and the mixing matrix $W$ is symmetric and doubly
stochastic with $\beta<1$ (see \eqref{beta} for the definition of $\beta$).
\end{enumerate}\end{assumption}

{Unlike
\cite{Duchi2012,Jakovetic2013,Nedic2009,Ram2010,Tsianos2012_strong},
{which {characterize} the ergodic convergence of
$f(\hat{x}_{(i)}(k))$ where
$\hat{x}_{(i)}(k)=\frac{1}{k}\sum_{s=0}^{k-1} {x}_{(i)}(s)$,   }
this paper establishes the {non-ergodic} convergence
of all local solution sequences $\{x_{(i)}(k)\}_{k\ge 0}$. In
addition,
the analysis in this paper does not assume bounded $\grad f_i$.
Instead, the following stepsize condition will ensure bounded
$\nabla f_i$: \beq\label{step_bnd} \alpha < O(1/L_h), \eeq where
$L_h = \max\{L_{f_1},\ldots,L_{f_n}\}$. This result is obtained
through interpreting the iteration \eqref{dec_grad} for all the
{agents} as a gradient descent iteration applied to a certain
Lyapunov function.}

{Under {Assumption \ref{assmp1}} and condition
\eqref{step_bnd}, the
 rate of $O(1/k)$ for ``near'' convergence is {shown}. Specifically, the
objective errors evaluated at the mean solution,
$f(\frac{1}{n}\sum_{i=1}^n x_{(i)}(k))-f^*$, and at any local
solution, $f({x}_{(i)}(k))-f^*$, both reduce at $O(1/k)$ until
reaching the level $O(\frac{\alpha}{1-\beta})$. The rate of the
mean solution is obtained by analyzing an inexact gradient descent
iteration, {somewhat} similar to
\cite{Duchi2012,Jakovetic2013,Nedic2009,Ram2010}. However, all of
their rates are given for the ergodic solution
$\hat{x}_{(i)}(k)=\frac{1}{k}\sum_{s=0}^{k-1} {x}_{(i)}(s)$. Our
rates are non-ergodic.}

In addition, a linear rate of ``near'' convergence is established if $f$ is
also strongly convex with modulus $\mu_f>0$, namely,
$$\langle\nabla f(x_a)-\nabla f(x_b), x_a-x_b \rangle
\ge \mu_f\|x_a-x_b\|^2,\quad \forall x_a,x_b\in\dom f,$$ or $f$ is
restricted strongly convex \cite{Lai2013} with modulus $\nu_f>0$,
\beq\label{rcvx}\langle\nabla f(x)-\nabla f(x^*), x-x^* \rangle
\ge \nu_f\|x-x^*\|^2,\quad \forall x\in\dom
f,~x^*=\Proj_{\cX^*}(x), \eeq where $\Proj_{\cX^*}(x)$ is the
projection of $x$ onto the  solution set $\cX^*$ and $\nabla
f(x^*)=0$. In both cases, we show that the mean solution error
$\|\frac{1}{n}\sum_{i=1}^n x_{(i)}(k)-x^*\|$ and the local
solution error $\|x_{(i)}(k)-x^*\|$  reduce geometrically until
reaching the level $O(\frac{\alpha}{1-\beta})$. Restricted
strongly convex functions are studied as they appear in the
applications of sparse optimization and statistical regression;
see \cite{ZhangYin2013} for some examples. The solution set
$\cX^*$ is a singleton if $f$ is strongly convex but not
necessarily so if $f$ is restricted strongly convex.

{Since our analysis uses a fixed stepsize, the local solutions will not be asymptotically consensual. To adapt our analysis to diminishing stepsizes, significant changes will be needed.}

Based on iteration  \eqref{dec_grad}, a decentralized algorithm is derived for the basis pursuit problem with distributed data to recover a sparse signal in Section
\ref{sec:3}. The algorithm converges linearly until reaching an $O(\frac{\alpha}{1-\beta})$-neighborhood of the  sparse signal.

Section \ref{sec:4} presents numerical results on the test
problems of decentralized least squares and decentralized basis
pursuit to verify our developed rates of convergence and the
levels of the landing neighborhoods.

Throughout the rest of this paper, we employ the following {notations} of stacked vectors:
$$[x_{(i)}]:=\begin{bmatrix}x_{(1)}\\x_{(2)}\\\vdots\\x_{(n)}\end{bmatrix}\in\RR^{np}\quad\text{and}\quad
h(k):=\begin{bmatrix}\nabla f_1(x_{(1)}(k))\\
\nabla f_2(x_{(2)}(k))\\\vdots\\ \nabla
f_n(x_{(n)}(k))\end{bmatrix}\in\RR^{np}.$$

\section{Convergence analysis}
\subsection{Bounded gradients}
\label{sec:2a}

Previous methods and analysis
\cite{Duchi2012,Jakovetic2013,Nedic2009,Ram2010,Tsianos2012_strong}
assume bound gradients or subgradients of $f_i$. The assumption
indeed plays a key role in the convergence analysis. For
decentralized gradient descent iteration \eqref{dec_grad}, it
gives \emph{bounded} deviation from mean $\|x_{(i)}(k) -
\frac{1}{n}\sum_{j=1}^n x_{(j)}(k)\|$. It is necessary in the convergence
analysis of subgradient methods, whether they are centralized or
decentralized. But as we show below,  the boundedness of $\nabla
f_i$  does not need to be guaranteed but is a consequence of bounded stepsize
$\alpha$, with dependence on the spectral properties of $W$. We
derive a tight bound on $\alpha$ for $\nabla f_i(x_{(i)}(k))$ to
be bounded.

\textbf{Example.} Consider $x\in\RR$ and a network formed by 3
connected agents (every pair of agents are directly linked).
Consider the following consensus optimization problem
$$ \Min_x\ \ f(x) = \sum_{i=1,2,3} f_i(x),\quad\text{where}~ f_i(x) = \frac{L_h}{2} (x-1)^2,$$
and $L_h>0$. This is a trivial average
consensus problem with $\nabla f_i(x_{(i)})=L_h(x_{(i)}-1)$ and $x^*=1$.
Take any $\tau \in (0,1/3)$ and let the mixing matrix be$$W =
\begin{bmatrix}1-2\tau & \tau & \tau\\ \tau & \tau & 1-2\tau\\
\tau & 1-2\tau & \tau\end{bmatrix},$$ which is symmetric doubly
stochastic. We have $\lambda_3(W)=3\tau - 1\in(-1,0)$. Start from
{$(x_{(1)},x_{(2)},x_{(3)})=(1,0,2)$}. Simple calculations
yield:
\begin{itemize}
\item if $\alpha < (1+\lambda_3(W))/L_h$, then $x_{(i)}(k)$
converges to $x^*$, $i=1,2,3$; (The consensus among $x_{(i)}(k)$
as $k\to\infty$ is due to design.) \item if $\alpha >
(1+\lambda_3(W))/L_h$, then $x_{(i)}(k)$ diverges and is
asymptotically unbounded where $i=1,2,3$; \item if $\alpha =
(1+\lambda_3(W))/L_h$, then $(x_{(1)}(k),x_{(2)}(k),x_{(3)}(k))$
equals $(1,2,0)$ at odd $k$ and $(1,0,2)$ at even $k$.
\end{itemize}
Clearly, if $x_{(i)}$ converges, then  $\nabla f_i(x_{(i)})$ converges and
thus stays bounded. In the above example $\alpha =
(1+\lambda_3(W))/L_h$ is the critical stepsize.

As each $\nabla f_i(x_{(i)})$ is Lipschitz continuous with
constant $L_{f_i}$,  $h(k)$ is Lipschitz continuous with constant
$$L_h =\max_{i}\{L_{f_i}\}.$$ We formally show that $\alpha <
(1+\lambda_n(W))/L_h$ ensures bounded $h(k)$. {The analysis
is based on the Lyapunov function
\beq\label{xi}\xi_\alpha([x_{(i)}]) :=
-\frac{1}{2}\sum_{i,j=1}^n w_{ij}x_{(i)}^Tx_{(j)}+\sum_{i=1}^n
\left(\frac{1}{2}\|x_{(i)}\|^2+\alpha f_i(x_{(i)})\right),\eeq
which is convex since all $f_i$ are convex and the
remaining terms $\frac{1}{2}\left(\sum_{i=1}^n \|x_{(i)}\|^2 -
\sum_{i,j=1}^n w_{ij}x_{(i)}^Tx_{(j)}\right)$ is also convex (and
uniformly nonnegative) due to {$\lambda_1(W)=1$}. In addition,
$\nabla \xi_{\alpha}$ is Lipschitz continuous with constant
$L_{\xi_\alpha}\le(1-\lambda_n(W))+{\alpha} L_h$. Rewriting
iteration \eqref{dec_grad} as
$$x_{(i)}(k+1) = \sum_{j=1}^n
w_{ij}x_{(j)}(k)-\alpha \nabla
f_i(x_{(i)}(k))=x_{(i)}(k)-\nabla_i\xi_{\alpha}([x_{(i)}(k)]),$$
we can observe that decentralized gradient
descent reduces to unit-stepsize centralized gradient
descent applied to minimize $\xi_\alpha([x_{(i)}])$.}

\begin{theorem}\label{h_bnd}
Under Assumption \ref{assmp1}, if the stepsize
\beq\label{stpbnd}\alpha \le (1+\lambda_n(W))/L_h,\eeq then,
starting from $x_{(i)}(0)=0$, ${i=1,2,\ldots,n}$, {{the}
sequence $x_{(i)}(k)$ generated by the iteration \eqref{dec_grad}
converges.} In addition we also have \beq\label{hk_bnd}
\|h(k)\|\le D:=\sqrt{2L_h  \left(\sum_{i=1}^n f_i(0) -{f^o}\right)}
\eeq for all $k=1,2,\ldots$, {where $f^o:=\sum_{i=1}^n f_i(x_{(i)}^o)$ and
$x_{(i)}^o=\arg\min_x f_i(x)$.}
\end{theorem}
\begin{proof}
{Note that the iteration \eqref{dec_grad} is equivalent to the
gradient descent iteration for the Lyapunov function \eqref{xi}.
From the classic analysis of gradient descent iteration in
\cite{bauschke2011convex} and \cite{Nesterov2007}, $[x_{(i)}(k)]$,
and hence $x_{(i)}(k)$, will converge to a certain point when
$\alpha \le (1+\lambda_n(W))/L_h$.}

Next we show \eqref{hk_bnd}. Since $\beta<1$, we have $\lambda_n(W)>-1$ and $(L_{\xi_\alpha}/2-
1)\le0$. Hence,
\begin{align*}\xi_\alpha([x_{(i)}(k+1)])
&\le \xi_\alpha([x_{(i)}(k)]) +\nabla\xi_\alpha([x_{(i)}(k)])^T([x_{(i)}(k+1) -x_{(i)}(k)])+\frac{L_{\xi_\alpha}}{2}\|[x_{(i)}(k+1) -x_{(i)}(k)]\|^2\\
&= \xi_\alpha([x_{(i)}(k)]) +(L_{\xi_\alpha}/2- 1)\|\nabla \xi_\alpha([x_{(i)}(k)])\|^2\\
&\le \xi_\alpha([x_{(i)}(k)]).
\end{align*}
Recall that $\frac{1}{2}\left(\sum_{i=1}^n \|x_{(i)}\|^2 -
\sum_{i,j=1}^n w_{ij}x_{(i)}^Tx_{(j)}\right)$ is nonnegative. Therefore,
we have \beq\label{fxibnd}\sum_{i=1}^n f_i(x_{(i)}(k))\le
\alpha^{-1}\xi_{\alpha}([x_{(i)}(k)])\le \cdots \le
\alpha^{-1}\xi_{\alpha}([x_{(i)}(0)])=\alpha^{-1}\xi_\alpha(0)=
\sum_{i=1}^n f_i(0). \eeq

On the other hand, for any differentiable convex function $g$ with the minimizer $x^*$
and Lipschitz constant $L_g$, we have $g(x_a)\ge g(x_b)+\nabla
g^T(x_b)(x_a-x_b)+\frac{1}{2L_g}\|\nabla g(x_a)-\nabla g(x_b)\|^2$
and $\nabla g(x^*)=0$.  Then, $\|\nabla g(x)\|^2\le 2
L_g(g(x)-g^*)$ where $g^*:=g(x^*)$. Applying this inequality and \eqref{fxibnd}, we
obtain
\begin{align}
\|h(k)\|^2 = \sum_{i=1}^n \|\nabla f_i(x_{(i)}(k))\|^2\le
{\sum_{i=1}^n
2L_{f_i} \left(f_i(x_{(i)}(k))-f_i^o\right)\le2L_h
\left(\sum_{i=1}^n f_i(0) -f^o\right),}
\end{align}
{where $f_i^o=f_i(x_{(i)}^o)$ and $x_{(i)}^o=\arg\min_x f_i(x)$. Note that $x_{(i)}^o$ exists because of Assumption \ref{assmp1}.
Besides, we denote $f^o=\sum_{i=1}^n  f_i^o$. }This completes
the proof.
\end{proof}

{In the above theorem, we choose $x_{(i)}(0)=0$ for
convenience. For general $x_{(i)}(0)$, a different bound for $\|h(k)\|$ can still
be obtained. Indeed, if $x_{(i)}(0)\neq 0$, then
$\alpha^{-1}\xi_\alpha(0)=\sum_{i=1}^n  f_i(0) +
\frac{1}{2\alpha} \big( \sum_{i=1}^n  \|x_{(i)}(0)\|^2 -
\sum_{i,j=1}^n w_{ij} x_{(i)}(0)^T x_{(j)}(0)\big)$ in \eqref{fxibnd}.
Hence we have $\|h(k)\|^2 \le 2L_h \big(\sum_{i=1}^n f_i(0)-f^o\big) +
\frac{L_h}{\alpha}\big(\sum_{i=1}^n \|x_{(i)}(0)\|^2 -
\sum_{i,j=1}^nw_{ij}x_{(i)}(0)^Tx_{(j)}(0)\big)$. The
initial values of $x_{(i)}(0)$ do not influence the stepsize
condition though they change the bound of gradient. For simplicity, we let
$x_{(i)}(0)=0$ in the rest of the paper.}

\textbf{Dependence on stepsize.} In \eqref{dec_grad}, the negative
gradient step $-\alpha\nabla f_i(x_{(i)})$ does not diminish at
$x_{(i)}=x^*$. Even if we let $x_{(i)}=x^*$ for all $i$, $x_{(i)}$
will immediately change once $\eqref{dec_grad}$ is applied.
Therefore, the term $-\alpha\nabla f_i(x_{(i)})$ prevents the
consensus of $x_{(i)}$. Even worse, because both terms in the
right-hand side of \eqref{dec_grad} change $x_{(i)}$, they can
possibly add up to an uncontrollable amount and cause $x_{(i)}(k)$
to diverge. The local averaging term is {stable itself}, so
the only choice we have is to limit the size of $-\alpha\nabla
f_i(x_{(i)})$ by bounding $\alpha$.

{\textbf{Network spectrum.} One can design $W$ so that
$\lambda_n(W) > 0$ and thus simply bound \eqref{stpbnd} to
$$\alpha \leq 1/L_h,$$ which no longer requires any spectral information of the underlying network. Given any mixing
matrix $\tilde{W}$ satisfying $1 = \lambda_1(\tilde{W})
> \lambda_2(\tilde{W}) \geq \cdots \geq \lambda_n(\tilde{W}) > -1$
(cf. \cite{Boyd2004}), one can design a new mixing matrix
$W=(\tilde{W}+I)/2$ that satisfies  $1 = \lambda_1(W)
> \lambda_2(W) \geq \cdots \geq \lambda_n(W) > 0$.
The same argument applies to the results throughout the paper.}

\subsection{Bounded deviation from mean}

Let $$\bar{x}(k) := \frac{1}{n}\sum_{i=1}^n x_{(i)}(k)$$ be the
\emph{mean} of $x_{(1)}(k),\ldots,x_{(n)}(k)$. We will later analyze the
error in terms of $\bar{x}(k)$ and then each $x_{(i)}(k)$. To enable
that analysis, we shall show that the deviation from mean
$\|x_{(i)}(k)-\bar{x}(k)\|$ is bounded uniformly over $i$ and $k$.
Then, any  bound of $\|\bar{x}(k)-x^*\|$ will  give a bound of
$\|x_{(i)}(k)-x^*\|$. Intuitively, if the deviation from mean is
unbounded,  then there would be no approximate consensus among
$x_{(1)}(k),\ldots,x_{(n)}(k)$. Without this approximate consensus,
descending individual $f_i(x_{(i)}(k))$ will not contribute to the
descent of $f(\bar{x}(k))$ and thus convergence is out of the
question. Therefore, it is critical to bound the deviation
$\|x_{(i)}(k)-\bar{x}(k)\|$.

\begin{lemma}
\label{lem:bnd_dev}
\label{bnd_dev} If \eqref{hk_bnd} holds and $\beta<1$, then the total deviation from mean is bounded, namely, $${{\|x_{(i)}(k) - \bar{x}(k)\| \le \frac{\alpha D}{1-\beta},\quad \forall k,\forall i.}}$$
\end{lemma}
\begin{proof}
Recall the definition of $[x_{(i)}]$ and $h(k)$, from the equation
\eqref{dec_grad} we have
$$
[x_{(i)}(k+1)]=(W\otimes I)[x_{(i)}(k)] - \alpha h(k),
$$
where $\otimes$ denotes the Kronecker product. From it, we
obtain\beq [x_{(i)}(k)]=-\alpha \sum_{s=0}^{k-1}(W^{k-1-s}\otimes
I)h(s). \eeq {Besides, letting
$[\bar{\mathbf{x}}(k)]=[\bar{x}(k);\cdots;\bar{x}(k)]\in \RR^{np}$, it
follows that
$$
[\bar{x}(k)]=\frac{1}{n}((1_n 1_n^T) \otimes I)) [\bar{\mathbf{x}}(k)].
$$}
As a result,
\begin{align}
\|x_{(i)}(k)-\bar{x}(k)\| & \leq \|[x_{(i)}(k)] - [\bar{\mathbf{x}}(k)] \| \nonumber\\
& = \|[x_{(i)}(k)]-\frac{1}{n}((1_n 1_n^T) \otimes I)) [x_{(i)}(k)]\| \nonumber\\
& = \|-\alpha \sum\limits_{s=0}^{k-1} (W^{k-1-s} \otimes I) h(s) +
                             \alpha \sum\limits_{s=0}^{k-1} \frac{1}{n} ((1_n 1_n^T W^{k-1-s}) \otimes I) h(s)\| \nonumber \\
                     &=  \|-\alpha \sum\limits_{s=0}^{k-1} (W^{k-1-s} \otimes I) h(s) +
                             \alpha \sum\limits_{s=0}^{k-1} \frac{1}{n} ((1_n 1_n^T) \otimes I) h(s)\|  \label{W_disappear}\\
                     &=  \alpha \| \sum\limits_{s=0}^{k-1} ((W^{k-1-s}-\frac{1}{n} 1_n 1_n^T) \otimes I) h(s) \| \nonumber \\
                     &\leq  \alpha \sum\limits_{s=0}^{k-1} \|W^{k-1-s}-\frac{1}{n} 1_n 1_n^T\| \| h(s)\| \nonumber \\
                     &=     \alpha \sum\limits_{s=0}^{k-1} \beta^{k-1-s} \| h(s)\|, \nonumber
\end{align}
where \eqref{W_disappear} holds since $W$ is doubly stochastic.
From $\|h(k)\|\le D$ and $\beta < 1$, it follows that
$$\|x_{(i)}(k)-\bar{x}(k)\| \leq \alpha \sum\limits_{s=0}^{k-1} \beta^{k-1-s} \| h(s)\| \leq \alpha \sum\limits_{s=0}^{k-1} \beta^{k-1-s} D \leq \frac{\alpha D}{1-\beta},$$
which completes the proof.
\end{proof}

{The proof of Lemma \ref{lem:bnd_dev} utilizes the spectral
property of the mixing matrix $W$. The constant in the upper bound is proportional to the stepsize $\alpha$ and
monotonically increasing with respect to the second largest
eigenvalue modulus $\beta$. The papers \cite{Duchi2012},
\cite{Nedic2009}, and \cite{Ram2010} also analyze the deviation of local solutions from their
mean, but their results are different.
The upper bound in \cite{Duchi2012} is given at the
termination time of the algorithm, which is not uniform in $k$. The two papers \cite{Nedic2009} and
\cite{Ram2010}, instead of bounding $\|W-\frac{1}{n}\mathbf{11}^T\|$, decompose it as the sum of element-wise $|w_{ij}-\frac{1}{n}|$ and then bounds it with the minimum nonzero element in
$W$.

}


{As discussed after Theorem \ref{h_bnd}, $D$ is affected by the
value of $x_{(i)}(0)$, if it is nonzero. In Lemma
\ref{lem:bnd_dev}, if $x_{(i)}(0)\neq 0$, then $[x_{(i)}(k)]=(W^k
\otimes I)[x_{(i)}(0)] -\alpha \sum_{s=0}^{k-1}(W^{k-1-s}\otimes
I)h(s)$. Substituting it into the proof of Lemma \ref{lem:bnd_dev}
we obtain
$$\|x_{(i)}(k)-\bar{x}(k)\| \leq \beta^k \|[x_{(i)}(0)]\| + \frac{\alpha D}{1-\beta}.$$
When $k \rightarrow \infty$, $\beta^k \|[x_{(i)}(0)]\| \rightarrow
0$ and, therefore, the last term dominates.}


A consequence of Lemma \ref{lem:bnd_dev} is that the distance
between the following two quantities is also bounded
\begin{align*}
g(k) & := \frac{1}{n}\sum_{i=1}^n  \nabla f_i(x_{(i)}(k)), \\
\bar{g}(k)&:= \frac{1}{n}\sum_{i=1}^n  \nabla f_i(\bar{x}(k)).
\end{align*}

{
\begin{lemma}
\label{bnd_g}
Under Assumption \ref{assmp1}, if \eqref{hk_bnd} holds and $\beta<1$, then
\begin{align*}
\|\nabla f_i(x_{(i)}(k))-\nabla f_i(\bar{x}(k))\|&\le \frac{\alpha DL_{f_i}}{1-\beta},\\
\|g(k)-\bar{g}(k)\|&\le \frac{\alpha DL_h}{1-\beta}.
\end{align*}
\end{lemma}
}
\begin{proof}
Assumption \ref{assmp1} gives
$$\|\nabla f_i(x_{(i)}(k))-\nabla f_i(\bar{x}(k))\| \le L_{f_i}\|x_{(i)}(k)-\bar{x}(k)\| \le \frac{\alpha DL_{f_i}}{1-\beta},$$
where the last inequality follows from Lemma \ref{lem:bnd_dev}. On the other hand, we have
$$\|g(k)-\bar{g}(k)\|=\|\frac{1}{n}\sum_{i=1}^n  \big(\nabla f_i(x_{(i)}(k)) - \nabla f_i(\bar{x}(k))\big)\|\le \frac{1}{n}\sum_{i=1}^n  L_{f_i}\|x_{(i)}(k)-\bar{x}(k)\|\le \frac{\alpha DL_h}{1-\beta},$$
which completes the proof.
\end{proof}

We are interested in $g(k)$ since $-\alpha g(k)$ updates the
average of $x_{(i)}(k)$. To see this, by taking the average of
\eqref{dec_grad} over $i$ and noticing $W=[w_{ij}]$ is doubly
stochastic, we obtain \beq\label{dec_avg}
\bar{x}(k+1)=\frac{1}{n}\sum_{i=1}^nx_{(i)}(k+1) =
\frac{1}{n}{\sum_{i,j=1}^n} w_{ij}x_{(j)} -
\frac{\alpha}{n}\sum_{i=1}^n \nabla
f_i(x_{(i)}(k))=\bar{x}(k)-\alpha g(k). \eeq On the other hand,
since the exact gradient of $\frac{1}{n}\sum_{i=1}^n
f_i(\bar{x}(k))$ is $\bar{g}(k)$, iteration \eqref{dec_avg} can be
viewed as an inexact gradient descent iteration (using $g(k)$
instead of $\bar{g}(k)$) for the problem
\beq\label{avg_prob}\Min_x~\bar{f}(x):=\frac{1}{n}\sum_{i=1}^n
f_i(x).\eeq It is easy to see that $\bar{f}$ is Lipschitz
continuous with the constant $$L_{\bar{f}}=\frac{1}{n}\sum_{i=1}^n
L_{f_i}.$$ If any $f_i$ is strongly convex, then so is $\bar{f}$,
with the modulus $\mu_{\bar{f}}=\frac{1}{n}\sum_{i=1}^n
\mu_{f_i}$.  Based on the above interpretation, next we bound
$f(\bar{x}(k))-f^*$ and $\|\bar{x}(k)-x^*\|$.

\subsection{Bounded distance to minimum}\label{sc:bdm}
We consider the convex,  restricted strongly convex, and strongly
convex cases. In the former two cases, the solution $x^*$ may be
non-unique, so we use the set of solutions $\cX^*$. We need the
followings for our analysis:
\begin{itemize}
\item  objective error $\bar{r}(k) : =
\bar{f}(\bar{x}(k))-\bar{f}^*=\frac{1}{n}(f(\bar{x}(k))-f^*)~\text{where}~
\bar{f}^*:=\bar{f}(x^*)$, $x^*\in\cX^*$; \item solution error $\bar{e}(k) :=
\bar{x}(k)-x^*(k)~\text{where}~x^*(k)=\Proj_{\cX^*}(\bar{x}(k))\in\cX^*.$
\end{itemize}

\begin{theorem}\label{r_cvg} Under Assumption \ref{assmp1}, if  $\alpha\le \min\{(1+\lambda_n(W))/L_h,1/L_{\bar{f}}\}=O(1/L_h)$, then while
$$\bar{r}(k) > C\sqrt{2}\cdot \frac{\alpha L_h D}{(1-\beta)}=O\left(\frac{\alpha}{1-\beta}\right)$$ (where constants $C$ and $D$ are defined in \eqref{def_c} and \eqref{hk_bnd}, respectively), the reduction of $\bar{r}(k)$ obeys
$$\bar{r}(k+1) \le \bar{r}(k) -O(\alpha\bar{r}^2(k)),$$
and therefore, $$\bar{r}(k)\le O\left(\frac{1}{\alpha k}\right).$$
In other words, $\bar{r}(k)$ decreases at a minimal rate of
$O(\frac{1}{\alpha k})=O(1/k)$ until reaching
$O(\frac{\alpha}{1-\beta})$.
\end{theorem}
\begin{proof}
{First we show that $\|\bar{e}(k)\| \leq C$. To this end, recall
the definition of $\xi_\alpha([x_{(i)}])$ in \eqref{xi}. Let
{$\tilde{\cX}$} denote its set of minimizer(s), which is
nonempty since each $f_i$ has a minimizer due to Assumption
\ref{assmp1}. Following the arguments in \cite[pp.
69]{Nesterov2007} and with the bound on $\alpha$, we have $d(k)\le
d(k-1)\le\cdots\le d(0)$, where $d(k):=\|[x_{(i)}(k)-
{\tilde{x}_{(i)}}]\|$ and {$[\tilde{x}_{(i)}]\in
\tilde{\cX}$}. Using $\|a_1+\cdots +a_n\|\le
\sqrt{n}\|{[a_1;\ldots;a_n]}\|$, we have
\begin{align}
\nonumber\|\bar{e}(k)\|&= \|\bar{x}(k)-x^*(k)\| = \|\frac{1}{n}\sum_{i=1}^n (x_{(i)}(k) - x^*)\| \le \frac{1}{\sqrt{n}}\|[x_{(i)}(k)-{x}^*]\|\\
\nonumber& \le \frac{1}{\sqrt{n}}(\|{[x_{(i)}(k)-\tilde{x}_{(i)}]}\|+\|{[\tilde{x}_{(i)} - x^*]}\|)\\
& \le \frac{1}{\sqrt{n}} (\|{[x_{(i)}(0) -
\tilde{x}_{(i)}]}\|+\|{[\tilde{x}_{(i)} -
x^*]}\|)=:C\label{def_c}
\end{align}

Next we show the convergence of $\bar{r}(k)$. By the assumption, we
have $1-\alpha L_{\bar{f}}\ge 0$, and thus
\begin{align*}
\bar{r}(k+1) &\le \bar{r}(k) +\langle \bar{g}(k),\bar{x}({k+1})-\bar{x}(k)\rangle +\frac{L_{\bar{f}}}{2}\|\bar{x}({k+1})-\bar{x}(k)\|^2\\
&\stackrel{\eqref{dec_avg}}{=} \bar{r}(k) - \alpha \langle \bar{g}(k),g(k)\rangle + \frac{\alpha^2 L_{\bar{f}}}{2}\|g(k)\|^2\\
&= \bar{r}(k) - \alpha \langle \bar{g}(k),\bar{g}(k)\rangle+ \frac{\alpha^2 L_{\bar{f}}}{2}\|\bar{g}(k)\|^2+ 2\alpha\frac{1-\alpha L_{\bar{f}}}{2} \langle \bar{g}(k),\bar{g}(k)-g(k)\rangle+\frac{\alpha^2 L_{\bar{f}}}{2}\|\bar{g}(k)-g(k)\|^2\\
&\le \bar{r}(k)-\alpha(1-\frac{\alpha L_{\bar{f}}}{2}-\delta\frac{1-\alpha L_{\bar{f}}}{2}) \|\bar{g}(k)\|^2+\alpha(\frac{\alpha L_{\bar{f}}}{2}+\delta^{-1}\frac{1-\alpha L_{\bar{f}}}{2})\|\bar{g}(k)-g(k)\|^2,
\end{align*}
where the last inequality follows from Young's inequality $\pm2a^Tb\le
\delta^{-1}\|a\|^2+ \delta\|b\|^2$ for any $\delta >0$. Although
we can later optimize over $\delta>0$, we simply take
$\delta = 1$. Since $\alpha\le (1+\lambda_n(W))/L_h$, we can apply
Theorem \ref{h_bnd} and then Lemma \ref{bnd_g} to the last term
above, and obtain
$$\bar{r}(k+1) \le \bar{r}(k) -\frac{\alpha}{2} \|\bar{g}(k)\|^2+\frac{\alpha^3 D^2L_h^2}{2(1-\beta)^2}.$$
Since $\|\bar{e}(k)\|\le C$ as shown in \eqref{def_c}, from $\bar{r}(k) =\bar{f}(\bar{x}(k))-\bar{f}^*\le \langle \bar{g}(k),\bar{x}(k)-x^*(k)\rangle=\langle \bar{g}(k),\bar{e}(k)\rangle$, we obtain that
$$\|\bar{g}(k)\|\ge \|\bar{g}(k)\|\frac{\|\bar{e}(k)\|}{C} \ge\frac{|\langle\bar{g}(k),\bar{e}(k)\rangle|}{C}\ge\frac{\bar{r}(k)}{C}, $$
which gives
$$\bar{r}(k+1) \le \bar{r}(k) -\frac{\alpha}{2C^2} \bar{r}^2(k)+\frac{\alpha^3 D^2L_h^2}{2(1-\beta)^2}.$$ Hence, while $\frac{\alpha}{2C^2} \bar{r}^2(k)> 2 \cdot \frac{\alpha^3 D^2L_h^2}{2(1-\beta)^2}$ or equivalently $\bar{r}(k) >  C\sqrt{2}\cdot \frac{\alpha L_h D}{(1-\beta)}$, we have $\bar{r}(k+1) \le \bar{r}(k)- O(\alpha\bar{r}^2(k))$. Dividing both sides by $\bar{r}(k)\bar{r}(k+1)$ gives $\frac{1}{\bar{r}(k)}+O(\frac{\alpha\bar{r}(k)}{\bar{r}(k+1)})\le \frac{1}{\bar{r}(k+1)}$. Hence, $\frac{1}{\bar{r}(k)}$ increase at $\Omega(\alpha k)$, or $\bar{r}(k)$ reduces at $O(1/(\alpha k))$, which completes the proof.}
\end{proof}

Theorem \ref{r_cvg} shows that until reaching
$f^*+O(\frac{\alpha}{1-\beta})$, $f(\bar{x}(k))$ reduces at the
rate of $O(1/(\alpha k))$. For fixed $\alpha$, there is a tradeoff
between the convergence rate and optimality. Again, upon the stopping
of iteration \eqref{dec_grad}, $\bar{x}(k)$ is not available to
any of  the agents but obtainable by invoking an average
consensus algorithm.

{
\begin{remark}
Since $\bar{f}(x)$ is convex, we have for all $i=1,2,\ldots,n$:
\begin{align}
\bar{f}(x_{(i)}(k)) - \bar{f}^* & \leq \bar{r}(k) + \langle \bar{g}(x_{(i)}(k)), x_{(i)}(k) - \bar{x}(k) \rangle \nonumber \\
& \leq \bar{r}(k) + \frac{1}{n}\sum_{j=1}^n \|\nabla f_j(x_{(i)}(k))\| \|x_{(i)}(k)) - \bar{x}(k)\| \nonumber \\
& \leq \bar{r}(k) + \frac{\alpha D^2}{1-\beta}. \nonumber
\end{align}
From Theorem \ref{r_cvg} we conclude that $\bar{f}(x_{(i)}(k)) -
\bar{f}^*$, like $\bar{r}(k)$, converges at $O(1/k)$ until
reaching $O(\frac{\alpha}{1-\beta})$.

This nearly sublinear convergence rate is stronger than those of
the distributed subgradient method \cite{Nedic2009} and the dual
averaging subgradient method \cite{Duchi2012}. Their rates are in
terms of objective error $f(\hat{x}_{(i)}(k))-f^*$ evaluated at the ergodic
solution
$\hat{x}_{(i)}(k)=\frac{1}{k}\sum_{s=0}^{k-1}x_{(i)}(s)$.
\end{remark}
}


Next, we bound $\|\bar{e}(k+1)\|$ under the assumption of restricted or standard strong convexities. To start, we
present a lemma.
\begin{lemma}\label{sc_bnd}Suppose that $\nabla\bar{f}$ is Lipschitz continuous with constant $L_{\bar{f}}$. Then, we have
$$\langle x-x^*, \grad\bar{f}(x) - \grad\bar{f}(x^*) \rangle \geq c_1\|\grad\bar{f}(x) - \grad\bar{f}(x^*)\|^2 +c_2 \|x-x^*\|^2$$
(where $x^*\in \cX^*$ and $\nabla\bar{f}(x^*)=0$) for the following cases:
\begin{enumerate}[a)]
\item (\cite[Theorem 2.1.12]{Nesterov2007}) if ${\bar{f}}$ is
strongly convex with modulus $\mu_{\bar{f}}$, then
$c_1=\frac{1}{\mu_{\bar{f}}+L_{\bar{f}}}$ and
$c_2=\frac{\mu_{\bar{f}}L_{\bar{f}}}{\mu_{\bar{f}}+L_{\bar{f}}}$;
\item (\cite[Lemma 2]{ZhangYin2013}) if ${\bar{f}}$ is restricted
strongly convex with modulus $\nu_{\bar{f}}$, then
$c_1=\frac{\theta}{L_{\bar{f}}}$ and $c_2=(1-\theta)\nu_{\bar{f}}$
for any $\theta\in[0,1]$.
\end{enumerate}
\end{lemma}

{{\begin{theorem} \label{mean convg} Under Assumption
\ref{assmp1}, if $f$ is either strongly convex with modulus
$\mu_f$ or restricted strongly convex with modulus $\nu_f$, and if
 $\alpha \le \min\{(1 + \lambda_n(W))/L_h, c_1\}=O(1/L_h)$
and $\beta<1$, then we have
$$\|\bar{e}(k+1)\|^2  \le c_3^2  \|\bar{e}(k)\|^2 + c_4^2,$$
where
$$c_3^2=1 - \alpha c_2+ \alpha\delta - \alpha^2 \delta c_2, \quad c_4^2=\alpha^3(\alpha+\delta^{-1}) \frac{L_h^2 D^2}{(1-\beta)^2},
\quad D=\sqrt{2L_h  \sum_{i=1}^n \left(f_i(0) -f_i^o\right)},$$ constants
$c_1$ and $c_2$ are given in Lemma \ref{sc_bnd},
$\mu_{\bar{f}}=\mu_{f}/n$ and $\nu_{\bar{f}}=\nu_{f}/n$, and
$\delta$ is any positive constant. In particular, if we set
$\delta=\frac{c_2}{2(1-\alpha c_2)}$ such that
$c_3=\sqrt{1-\frac{\alpha c_2}{2}} \in (0,1)$, then we have
$$\|\bar{e}(k)\|\le c_3^{k} \|\bar{e}(0)\|+O(\frac{\alpha}{1-\beta}).$$
\end{theorem}}}
\begin{proof}
Recalling that $x^*(k+1)=\Proj_{\cX^*}(\bar{x}(k+1))$ and
$\bar{e}(k+1)=\bar{x}(k+1)-x^*(k+1)$, we have
\begin{align*}
\|\bar{e}(k+1)\|^2  & \leq \|\bar{x}(k+1) - x^*(k)\|^2 \nonumber\\
& = \|\bar{x}(k) - x^*(k) - \alpha g(k)\|^2 \nonumber \\
& = \|\bar{e}(k)-\alpha \bar{g}(k) + \alpha(\bar{g}(k)-g(k))\|^2 \quad  \nonumber \\
& = \|\bar{e}(k)-\alpha \bar{g}(k)\|^2 +\alpha^2 \|\bar{g}(k)-g(k)\|^2 + 2\alpha (\bar{g}(k)-g(k))^T(\bar{e}(k)-\alpha \bar{g}(k)) \nonumber \\
& \leq (1 + \alpha\delta) \|\bar{e}(k)-\alpha \bar{g}(k)\|^2 + \alpha(\alpha+\delta^{-1}) \|\bar{g}(k)-g(k)\|^2,
\end{align*}
where the last inequality follows again from $\pm2a^Tb\le
\delta^{-1}\|a\|^2+ \delta\|b\|^2$ for any $\delta >0$. The bound
of $\|\bar{g}(k)-g(k)\|^2$ follows from Lemma \ref{bnd_g} and
Theorem \ref{h_bnd}, and we shall bound $ \|\bar{e}(k)-\alpha
\bar{g}(k)\|^2$, which is a standard exercise; we repeat below for
completeness. Applying Lemma \ref{sc_bnd} and noticing
$\bar{g}(x)=\nabla{\bar{f}}(x)$ by definition, we have
\begin{align*}
      \|\bar{e}(k)-\alpha \bar{g}(k)\|^2 &
=     \|\bar{e}(k)\|^2 + \alpha^2 \|\bar{g}(k)\|^2 - 2\alpha \bar{e}(k)^T \bar{g}(k) \nonumber \\
&\le    \|\bar{e}(k)\|^2 + \alpha^2 \|\bar{g}(k)\|^2 - \alpha c_1 \|\bar{g}(k)\|^2-\alpha c_2\|\bar{e}(k)\|^2 \nonumber \\
& = (1-\alpha c_2) \|\bar{e}(k)\|^2 + \alpha(\alpha -c_1) \|\bar{g}(k)\|^2. \end{align*}
We shall pick $\alpha\le c_1$ so that $\alpha(\alpha -c_1) \|\bar{g}(k)\|^2\le 0$. Then from the last two inequality arrays, we have
\begin{align*}
\|\bar{e}(k+1)\|^2  & \leq (1 + \alpha\delta)(1-\alpha c_2) \|\bar{e}(k)\|^2 +\alpha(\alpha+\delta^{-1}) \|\bar{g}(k)-g(k)\|^2 \\
& \leq (1 - \alpha c_2+ \alpha\delta - \alpha^2 \delta c_2)
\|\bar{e}(k)\|^2 +\alpha^3(\alpha+\delta^{-1}) \frac{L_h^2
D^2}{(1-\beta)^2}.
\end{align*}

Note that if $f$ is strongly convex, then $c_1c_2=\frac{\mu_{\bar{f}} L_{\bar{f}}}{(\mu_{\bar{f}} + L_{\bar{f}})^2} < 1$; if $f$
is restricted strongly convex, then $c_1c_2=\frac{\theta(1-\theta)\nu_{\bar{f}}}{L_{\bar{f}}} < 1$ because $\theta \in [0,1]$ and
$\nu_{\bar{f}}<L_{\bar{f}}$. Therefore we have $c_1 < 1/c_2$. When $\alpha < c_1$, $(1 + \alpha\delta)(1-\alpha c_2) > 0$.

Next, since
$$\|\bar{e}(k)\|^2\le c_3^{2k} \|\bar{e}(0)\|^2+\frac{1-c_3^{2k}}{1-c_3^2}c_4^2\le c_3^{2k} \|\bar{e}(0)\|^2+\frac{c_4^2}{1-c_3^2},$$
we get
$$\|\bar{e}(k)\| \leq c_3^{k} \|\bar{e}(0)\| + \frac{c_4}{\sqrt{1-c_3^2}}.$$
If we set $$\delta=\frac{c_2}{2(1-\alpha c_2)},$$ then we obtain
$$c_3^2=1-\frac{\alpha c_2}{2}<1,$$
$$\frac{c_4}{\sqrt{1-c_3^2}}=\frac{\alpha L_h D}{1-\beta}\sqrt{\frac{\alpha(\alpha+\frac{2(1-\alpha c_2)}{c_2})}{\frac{\alpha c_2}{2}}}=\frac{\alpha L_h D}{1-\beta} \sqrt{\frac{4}{c_2^2}-\frac{2}{c_2}\alpha}=O(\frac{\alpha}{1-\beta}),$$
which completes the proof.
\end{proof}

\begin{remark}
As a result, if $f$ is strongly convex, then $\bar{x}(k)$
geometrically converges until reaching an $O(\frac{\alpha}{1-\beta})$-neighborhood
of the unique solution $x^*$; on the other hand, if $f$ is
restricted strongly convex, then $\bar{x}(k)$ geometrically
converges until reaching an $O(\frac{\alpha}{1-\beta})$-neighborhood of the
 solution set $\mathcal{X}^*$.
\end{remark}

\subsection{Local agent convergence}
\begin{corollary} \label{coro2}
Under Assumption \ref{assmp1}, if $f$ is either strongly convex or
restricted strongly convex, $\alpha < \min\{(1 +
\lambda_n(W))/L_h, c_1\}$ and $\beta<1$,
then we have
$$\|x_{(i)}(k)-{x}^*(k)\| \le c_3^k  \|{x}^*(0)\| + \frac{c_4}{\sqrt{1-c_3^2}} + \frac{\alpha D}{1-\beta},$$
where $x^*(0),{x}^*(k)\in\cX^*$ are solutions defined at the beginning of subsection \ref{sc:bdm} and {the constants $c_3$, $c_4$, $D$ are the same as given in Theorem \ref{mean convg}.}
\end{corollary}
\begin{proof}
From Lemma \ref{lem:bnd_dev} and Theorem \ref{mean convg} we have
\begin{align}
&    \|x_{(i)}(k)-{x}^*(k)\| \nonumber \\
\leq&\|\bar{x}(k)-{x}^*(k)\|+\|x_{(i)}(k)-\bar{x}(k)\| \nonumber \\
\leq& c_3^k\|{x}^*(0)\|+\frac{c_4}{\sqrt{1-c_3^2}}+\frac{\alpha
D}{1-\beta}, \nonumber
\end{align}
which completes the proof.
\end{proof}

\begin{remark}
Similar to Theorem \ref{mean convg} and Remark 1, if we set
$\delta=\frac{c_2}{2(1-\alpha c_2)}$, and if $f$ is strongly
convex, then {$x_{(i)}(k)$} geometrically converges to an
$O(\frac{\alpha}{1-\beta})$-neighborhood of the unique solution
$x^*$; if $f$ is restricted strongly convex, then
{$x_{(i)}(k)$} geometrically converges to an
$O(\frac{\alpha}{1-\beta})$-neighborhood of the solution set
$\mathcal{X}^*$.
\end{remark}

\section{Decentralized basis pursuit}
\label{sec:3}

\subsection{Problem statement}

We derive an algorithm for solving a decentralized basis pursuit
problem to illustrate the application of iteration
\eqref{dec_grad}.

Consider a multi-agent network of $n$  agents who collaboratively
find a sparse representation $y$ of a given signal $b \in \RR^p$
that is known to all the agents. Each agent $i$ holds a part $A_i
\in \RR^{p \times q_i}$ of the entire dictionary $A \in \RR^{p
\times q}$, where $q = \sum_{i=1}^n q_i$, and shall recover the
corresponding  $y_i \in \RR^{q_i}$. Let
\begin{equation}
y := \left[
             \begin{array}{c}
               y_1 \\
               \vdots \\
               y_n
             \end{array}
           \right] \in \RR^{q},\quad
A := \left[
             \begin{array}{ccc}
              |   &        & | \\
              A_1 & \hdots & A_n \\
              |   &        & |
             \end{array}
           \right] \in \RR^{p \times q}. \nonumber
\end{equation}
 The problem is
\begin{align} \label{eq:cp-bp}
\Min \limits_y &  \quad \|y\|_1, \\
\st & \quad \sum_{i=1}^n A_i y_i = b, \nonumber
\end{align}
where $\sum_{i=1}^n A_i y_i=Ay$. {This formulation is a column-partitioned version of decentralized basis pursuit, as opposed to the row-partitioned version in \cite{Mota2012} and \cite{yuan2013}. Both versions} find applications
in, for example, collaborative spectrum sensing
\cite{Bazerque2010}, sparse event detection \cite{meng2009sparse},
and seismic modeling \cite{Mota2012}.


Developing efficient decentralized algorithms to solve
\eqref{eq:cp-bp} is nontrivial since the objective function is
neither differentiable nor strongly convex, and the constraint
couples all the agents. In this paper, we turn to an equivalent
and tractable reformulation by appending a strongly convex term
and solving its Lagrange dual problem by decentralized gradient
descent. Consider the augmented form of \eqref{eq:cp-bp} motivated
by \cite{Lai2013}:
\begin{align} \label{eq:cp-bp-lb}
\Min \limits_y &  \quad \|y\|_1 + \frac{1}{2\gamma} \|y\|^2, \\
\st & \quad A y = b, \nonumber
\end{align}
where the regularization parameter $\gamma > 0$ is chosen so that
\eqref{eq:cp-bp-lb} returns a solution to \eqref{eq:cp-bp}.
Indeed, provided that $Ay=b$ is consistent, there always exists $\gamma_{\min}
> 0$ such that the solution to \eqref{eq:cp-bp-lb} is also a
solution to \eqref{eq:cp-bp} for any $\gamma \geq \gamma_{\min}$
\cite{Friedlander2007,Yin2010}. Linearized Bregman iteration
proposed in \cite{yin2008bregman} is proven to converge to the
unique solution of \eqref{eq:cp-bp-lb} efficiently. See
\cite{Yin2010} for its analysis and \cite{osher2011fast} for
important improvements. Since the problem \eqref{eq:cp-bp-lb} is
now solved over a network of agents, we need to devise a
decentralized version of linearized Bregman iteration.

The Lagrange dual of \eqref{eq:cp-bp-lb}, casted as a minimization
(instead of maximization) problem, is
\begin{align}
   \Min\limits_x~f(x) := \frac{\gamma}{2} \|A^T x - \text{Proj}_{[-1,1]}(A^T x)\|^2 - b^T
   x, \label{eq:dual}
\end{align}
where $x \in \RR^p$ is the dual variable and
$\text{Proj}_{[-1,1]}$ denotes the element-wise projection onto $[-1,1]$.

We turn \eqref{eq:dual} into the form of \eqref{eq:original
problem}: \beq\label{dec_bp} \Min_x~f(x)=\sum_{i=1}^n
f_i(x),~\text{where}~f_i(x):= \frac{\gamma}{2} \|A_i^T x -
\text{Proj}_{[-1,1]}(A_i^T x)\|^2 - \frac{1}{n} b^T x. \eeq
The function $f_i$ is defined with $A_i$ and $b$, where matrix $A_i$ is the
private information of agent $i$. The local objective functions
$f_i$ are differentiable with the gradients given as
\begin{align} \label{eq:gradient}
\nabla f_i(x) = \gamma A_i \text{Shrink}(A_i^T x) - \frac{b}{n},
\end{align}
where $\text{Shrink}(z)$ is the shrinkage operator defined as
$\max(|z|-1,0)\sign(z)$ component-wise.

Applying the iteration \eqref{dec_grad} to the problem \eqref{dec_bp}
starting with $x_{(i)}(0) = 0$, we obtain the iteration
\begin{align} \label{eq:dlb}
\boxed{x_{(i)}(k+1) = \sum_{j=1}^n w_{ij} x_{(j)}(k) - \alpha \left(A_i y_i(k) - \frac{b}{n}\right), \quad \text{where}~ y_i(k) = \gamma
\text{Shrink}(A_i^T x_{(i)}(k)).}
\end{align}
{Note that the primal solution $y_i(k)$ is
iteratively updated, as a middle step for the update of
$x_{(i)}(k+1)$.}

{ It is easy to verify that the local objective functions
$f_i$ are Lipschitz differentiable with the constants $L_{f_i} =
\gamma \| A_i\|^2$. Besides, given that $Ay=b$ is consistent,
\cite{Lai2013} proves that $f(x)$ is restricted strongly convex
with a computable constant $\nu_f >0$.
Therefore,  the objective function $f(x)$ in \eqref{eq:dual} has
$L_h=\max\{\gamma \|A_i\|^2:i=1,2,\cdots,n\}$,
$L_{\bar{f}}=\frac{\gamma}{n} \sum_{i=1}^n \|A_i\|^2$ and
$\nu_{\bar{f}}=\nu_{f}/n $. By Theorem \ref{mean convg}, any
local dual solution $x_{(i)}(k)$ generated by iteration
\eqref{eq:dlb} linearly converges to a neighborhood of the
solution set of \eqref{eq:dual}, and the primal solution $y(k) =
[y_1(k); \cdots; y_n(k)]$  linearly converges to a neighborhood of
the unique solution of \eqref{eq:cp-bp-lb}.}

\begin{theorem}\label{BP-convg}
Consider $x_{(i)}(k)$ generated by iteration \eqref{eq:dlb} and
$\bar{x}(k) := \frac{1}{n}\sum_{i=1}^n x_{(i)}(k)$. The unique
solution of \eqref{eq:cp-bp-lb} is $y^*$ and the projection of
$\bar{x}(k)$ onto the optimal solution set of \eqref{eq:dual} is
$\bar{x}^*(k)=\text{Proj}_{\mathcal{X}^*}(\bar{x}(k))$. If the
stepsize $\alpha < \min\{{(1 + \lambda_n(W))}/{L_h}, c_1\}$, we
have
\begin{align}
\label{thm4}
\|x_{(i)}(k)-\bar{x}^*(k)\| \le c_3^k  \|\bar{x}^*(0)\| + \left(\frac{c_4}{\sqrt{1-c_3^2}} + \frac{\alpha D}{1-\beta}\right),
\end{align}
where the constants $c_3$ and $c_4$ are the same as given in
Theorem \ref{mean convg}. In particular, if we set
$\delta=\frac{c_2}{2(1-\alpha c_2)}$ such that
$c_3=\sqrt{1-\frac{\alpha c_2}{2}} \in (0,1)$, then  $\frac{c_4}{\sqrt{1-c_3^2}} + \frac{\alpha D}{1-\beta}=O(\frac{\alpha}{1-\beta})$.
On the other hand, the primal solution satisfies
\begin{align}
\label{ybound} \| y(k) - y^* \| \leq  n \gamma \max_{i} \left(
\|A_i\| \|x_{(i)}(k) - \bar{x}^*(k)\| \right).
\end{align}
\end{theorem}

\begin{proof}
The result \eqref{thm4} is a corollary of Corollary
\ref{coro2}. We focus on showing \eqref{ybound}.

Given any dual solution $\bar{x}(k)$, the primal solution of
\eqref{eq:cp-bp-lb} is $y^* = \gamma \text{Shrink}(A^T
\bar{x}^*(k))$. Recall that $y(k) = [y_1(k); \cdots;
y_n(k)]$ and $y_i(k) = \gamma \text{Shrink}(A_i^T x_{(i)}(k))$. We have
\begin{align} \label{eq:ybound-1}
\| y(k) - y^* \| = & \| [\gamma \text{Shrink}(A_1^T x_{(1)}(k));
\cdots; \gamma \text{Shrink}(A_n^T x_{(n)}(k))] - \gamma
\text{Shrink}(A^T
\bar{x}^*(k)) \| \\
                 \leq & \gamma \sum_{i=1}^n \| \text{Shrink}(A_i^T x_{(i)}(k)) - \text{Shrink}(A_i^T \bar{x}^*(k)) \|. \nonumber
\end{align}
Due to the contraction of the shrinkage operator, we have the bound $\|
\text{Shrink}(A_i^T x_{(i)}(k)) - \text{Shrink}(A_i^T \bar{x}^*(k)) \|
\leq \|A_i\| \|x_{(i)}(k) - \bar{x}^*(k)\| \leq \max_{i} \left( \|A_i\|
\|x_{(i)}(k) - \bar{x}^*(k)\| \right)$. Combining this inequality with
\eqref{eq:ybound-1}, we get \eqref{ybound}.
\end{proof}

\section{Numerical experiments}
\label{sec:4} In this section, we report our numerical results
applying the iteration \eqref{dec_grad} to a decentralized least
squares problem and the iteration \eqref{eq:dlb} to a
decentralized basis pursuit problem.

We generate a network consisting of $n$ agents with
$\frac{n(n-1)}{2}\eta$ edges that are uniformly randomly chosen, where  $n=100$ and $\eta=0.3$ are chosen for all the tests.
We ensure a connected network.

\subsection{Decentralized gradient descent for least squares}
\label{sec:4a} We apply the iteration \eqref{dec_grad} to the
least squares problem
\begin{align}
   \Min\limits_{x\in\RR^3} \quad \frac{1}{2}\|b-Ax\|^2=\sum\limits_{i=1}^n\frac{1}{2}\|b_i-A_ix\|^2. \label{LS}
\end{align}
The entries of the true signal $x^*\in\RR^3$ are i.i.d samples
from the Gaussian distribution $\mathcal{N}(0,1)$. $A_i \in
\RR^{3\times3}$ is the linear sampling matrix of agent $i$ whose
elements are i.i.d samples from  $\mathcal{N}(0,1)$, and
$b_i=A_ix^* \in \RR^3$ is the measurement vector of agent $i$.

For the problem \eqref{LS}, let $f_i(x)=\frac{1}{2}\|b_i-A_ix\|^2$. For any $x_a,\
x_b \in \RR^3$, $\|\nabla f_i(x_a)-\nabla
f_i(x_b)\|=\|A_i^TA_i(x_a-x_b)\|\leq \|A_i^TA_i\|\|x_a-x_b\|$, so
 $\nabla f_i(x)$ is Lipschitz continuous. In addition,
$\frac{1}{2}\|b-Ax\|_2^2$ is strongly convex since $A$ has full
column rank, with probability 1.

Fig. \ref{fig:1} depicts the convergence of the error $\bar{e}(k)$
corresponding to five different stepsizes. It shows that
$\bar{e}(k)$ reduces linearly until reaching an
$O(\alpha)$-neighborhood, which agrees with Theorem \ref{mean
convg}. Not surprisingly, a smaller $\alpha$  causes the algorithm
to converge more slowly.


\begin{figure}
\centering
\begin{center}
\includegraphics[height=7cm]{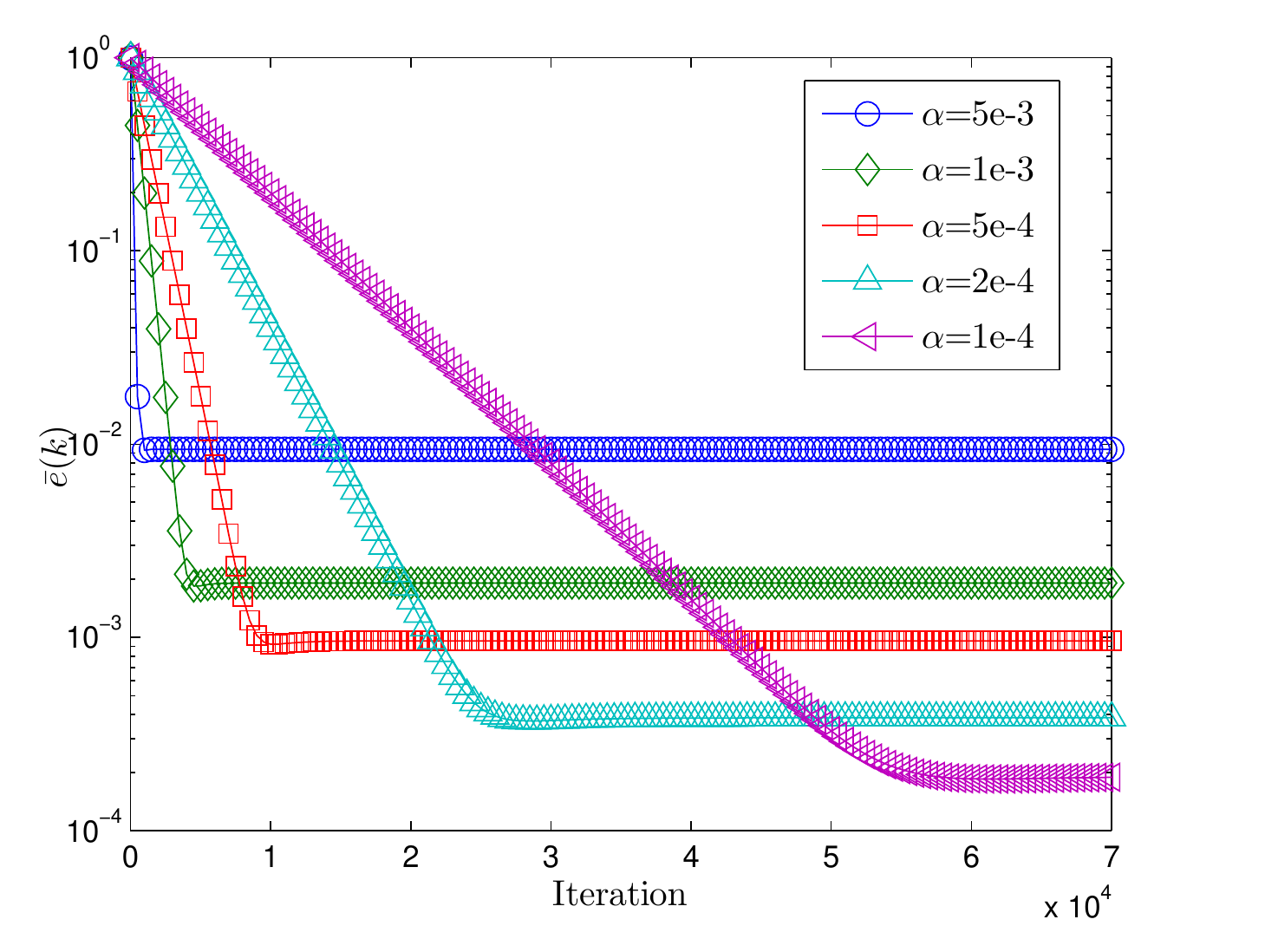}
\caption{Comparison of different fixed stepsizes for the decentralized gradient descent
algorithm.} \label{fig:1}
\end{center}
\end{figure}

%
%
%

Fig. \ref{fig:2} compares our theoretical stepsize bound in Theorem \ref{h_bnd} to the empirical bound of $\alpha$.
The theoretical bound for this
experimental network is
$\min\{\frac{1+\lambda_n(W)}{L_h},c_1\}=0.1038$. In Fig. 2, we
choose $\alpha=0.1038$ and then the slightly larger $\alpha=0.12$.
We observe  convergence
 with $\alpha=0.1038$ but clear divergence with $\alpha=0.12$. This shows that our  bound on $\alpha$ is quite close to the actual requirement.
\begin{figure}
\centering
\begin{center}
\includegraphics[height=7cm]{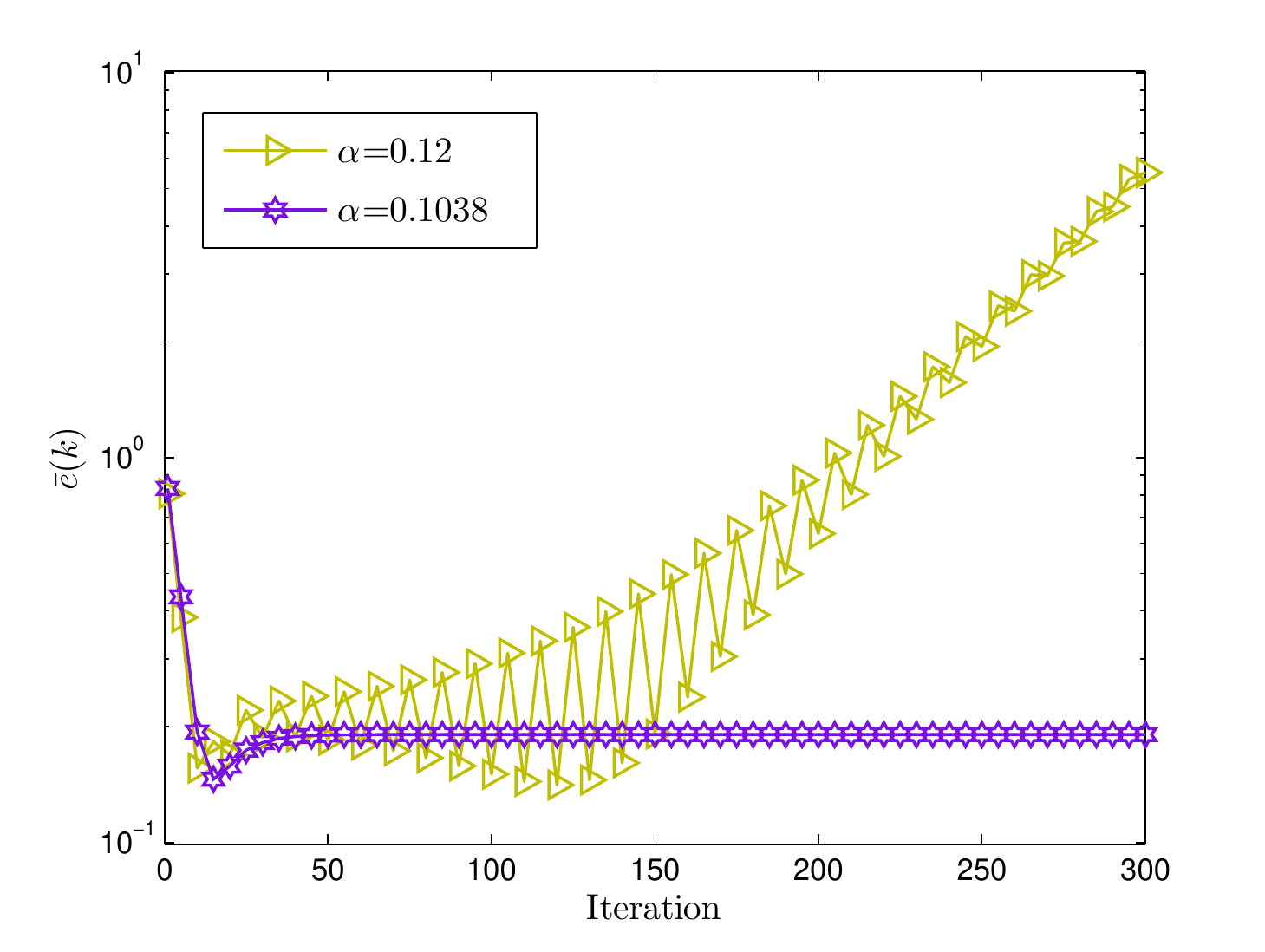}
\caption{Comparison of the decentralized gradient descent
algorithm with stepsizes $\alpha=0.1038$ and $\alpha=0.12$.
} \label{fig:2}
\end{center}
\end{figure}

\subsection{Decentralized gradient descent for basis pursuit}
\label{sec:4b} In this subsection we test the iteration \eqref{eq:dlb} for the decentralized basis pursuit problem
\eqref{eq:cp-bp}.

Let $y \in \RR^{100}$ be the unknown signal  whose entries are i.i.d. samples from
$\mathcal{N}(0,1)$. The entries of the measurement matrix $A \in
\RR^{50\times100}$ are also i.i.d. samples from $\mathcal{N}(0,1)$. Each
agent $i$ holds the $i$th column of $A$. $b=Ay\in \RR^{50}$ is the measurement
vector. We use the same network as in the last test.

\begin{figure}
\centering
\begin{center}
\includegraphics[height=7cm]{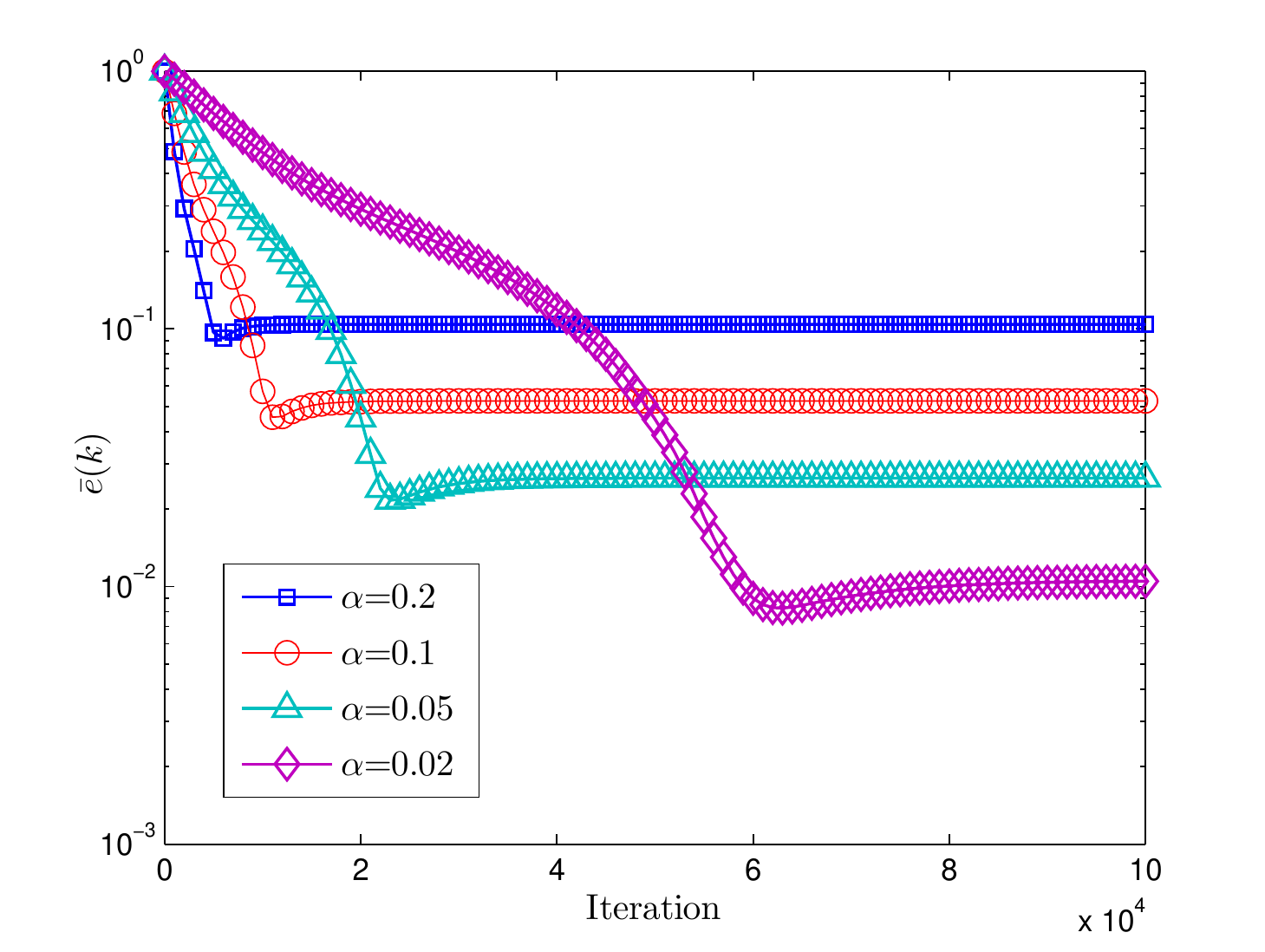}
\caption{Convergence of the mean value of the dual variable
$\bar{x}(k)$.}
\label{fig:3}
\end{center}
\end{figure}

\begin{figure}
\centering
\begin{center}
\includegraphics[height=7cm]{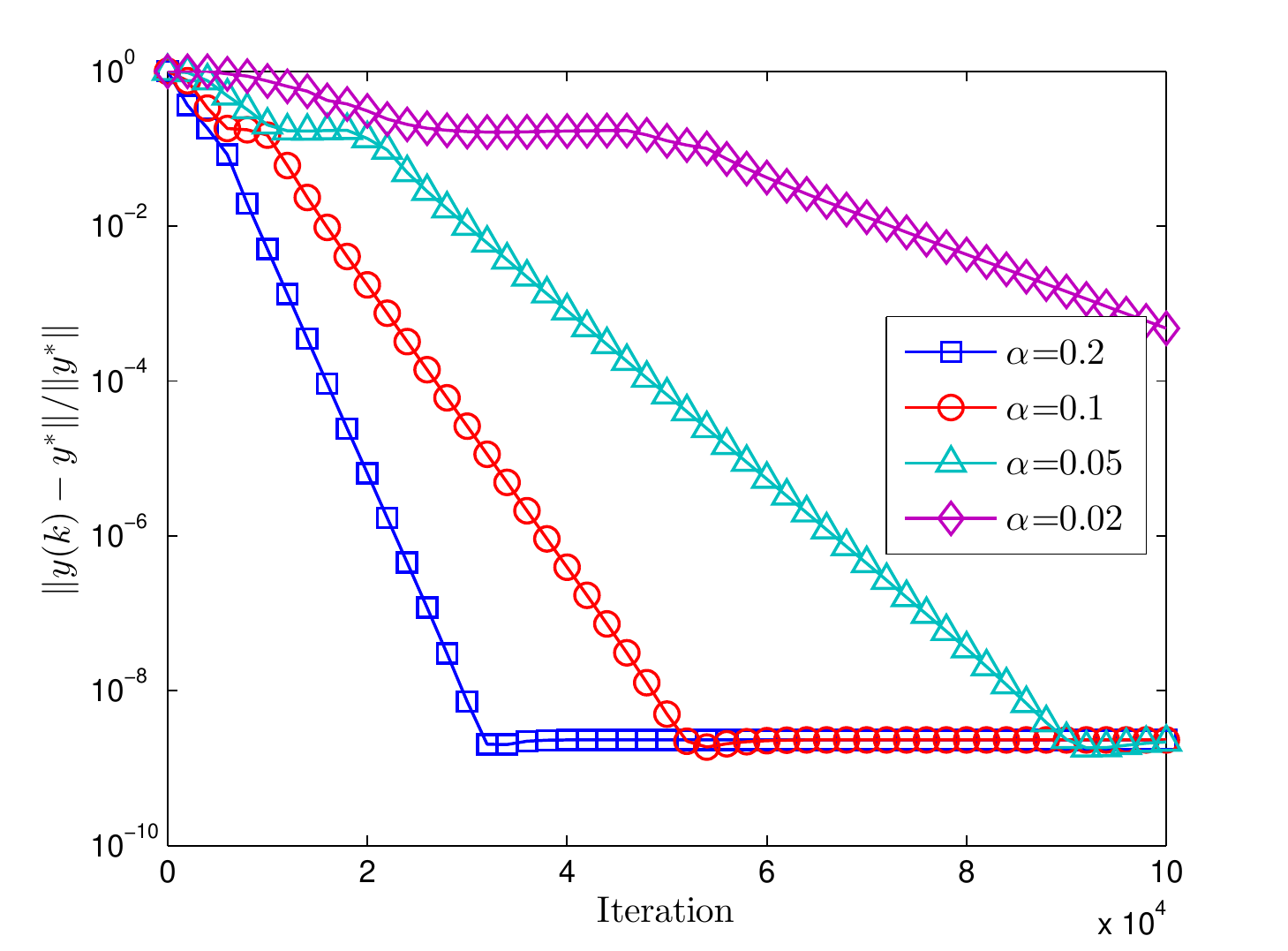}
\caption{Convergence of the primal variable $y(k)$. $y^*$ is the
 solution of the problem \eqref{eq:cp-bp-lb}.
} \label{fig:4}
\end{center}
\end{figure}

Fig. \ref{fig:3} depicts the convergence of  $\bar{x}(k)$, the
mean of the dual variables at iteration $k$. As stated in Theorem
\ref{BP-convg}, $\bar{x}(k)$ converges linearly to an
$O(\alpha)$-neighborhood of the  solution set $\mathcal{X}^*$. The
limiting errors $\bar{e}(k)$ corresponding to the four values of
$\alpha$ are proportional to $\alpha$. As the stepsize becomes
smaller, the algorithm converges more accurately to
$\mathcal{X}^*$. Fig. \ref{fig:4} shows the linear convergence of
the primal variable $y(k)$. It is interesting that the $y(k)$
corresponding to three different values of $\alpha$ appear to
reach the same level of accuracy, which might be related to the
error forgetting property of the first-order $\ell_1$ algorithm
\cite{YinOsher2013} and deserves further investigation.

\section{Conclusion}
Consensus optimization problems in multi-agent networks arise in
applications such as mobile computing, self-driving cars'
coordination, cognitive radios, as well as collaborative data
mining. Compared to the traditional centralized approach, a
decentralized  approach offers more balanced communication load
and better privacy protection. In this paper, our effort is to
provide a mathematical understanding to the decentralized gradient
descent method with a fixed stepsize. We give a tight
condition for guaranteed convergence, as well as an example to
illustrate the fail of convergence when the condition is violated.
We provide the analysis of convergence and the rates of convergence
for problems with different properties and establish the relations
between network topology, stepsize, and convergence speed, which
shed some light on network design. The numerical observations
reasonably matches the theoretical results.

\section*{Acknowledgements}
Q. Ling is supported by NSFC grant 61004137. W. Yin is supported
by ARL and ARO grant W911NF-09-1-0383 and NSF grants DMS-0748839
and DMS-1317602. The authors thank Yangyang Xu for helpful
comments.

\bibliographystyle{siam}
\bibliography{DGD}

\end{document}